\documentclass[letterpaper,11pt]{amsart}

\usepackage[margin=1.2in]{geometry}
\usepackage{amsmath,amsthm,amssymb}
\usepackage{xspace,xcolor}
\usepackage[breaklinks,colorlinks,citecolor=teal,linkcolor=teal,urlcolor=teal,pagebackref,hyperindex]{hyperref}
\usepackage[alphabetic]{amsrefs}
\usepackage[all]{xy}

\theoremstyle{plain}
\newtheorem{thm}{Theorem}[section]

\newtheorem{lem}[thm]{Lemma}
\newtheorem{prop}[thm]{Proposition}
\newtheorem{cor}[thm]{Corollary}
\newtheorem{conj}[thm]{Conjecture}

\theoremstyle{definition}
\newtheorem{defi}[thm]{Definition}

\newtheorem{eg}[thm]{Example}

\newtheorem{question}[thm]{Question}

\theoremstyle{remark}
\newtheorem{rmk}[thm]{Remark}

\def\Z{{\mathbb Z}}
\def\Q{{\mathbb Q}}
\def\R{{\mathbb R}}
\def\C{{\mathbb C}}

\def\P{{\mathbb P}}

\def\cB{\mathcal{B}}

\def\L{\mathcal{L}}
\def\M{\mathcal{M}}
\def\O{\mathcal{O}}
\def\U{\mathcal{U}}
\def\X{\mathcal{X}}

\def\fa{\mathfrak{a}}

\def\fn{\mathfrak{n}}

\def\a{\alpha}

\def\g{\gamma}

\def\f{\phi}

\def\ep{\epsilon}

\def\l{\lambda}
\def\n{\nu}

\def\om{\omega}
\def\p{\pi}
\def\r{\rho}

\def\t{\tau}
\def\x{\xi}

\def\D{\Delta}
\def\G{\Gamma}

\def\Om{\Omega}

\def\.{\cdot}
\def\^{\widehat}
\def\~{\widetilde}

\def\ov{\overline}

\def\rat{\dashrightarrow}

\def\inj{\hookrightarrow}

\def\de{\partial}

\def\({\left(}
\def\){\right)}

\renewcommand{\and}{ \ \ \text{ and } \ \ }
\renewcommand{\for}{ \ \ \text{ for } \ \ }

\def\reg{\mathrm{reg}}

\def\Jac{\mathrm{Jac}}
\def\num{\mathrm{num}}
\def\st{\mathrm{st}}
\def\KR{\mathrm{KR}}

\DeclareMathOperator{\codim} {codim}

\DeclareMathOperator{\Gr} {Gr}

\DeclareMathOperator{\Sing} {Sing}

\DeclareMathOperator{\Ex} {Ex}

\DeclareMathOperator{\ord} {ord}

\DeclareMathOperator{\Supp} {Supp}
\DeclareMathOperator{\dirlim} {\varinjlim}

\DeclareMathOperator{\id} {id}
\DeclareMathOperator{\Fitt} {Fitt}

\DeclareMathOperator{\lSFT} {lSFT}

\DeclareMathOperator{\CZ} {CZ}

\DeclareMathOperator{\mld} {mld}

\DeclareMathOperator{\hmi} {hmi}
\DeclareMathOperator{\mi} {mi}

\title{Towards a link theoretic characterization of smoothness}

\author{Tommaso de Fernex}

\address{Department of Mathematics, University of Utah, Salt Lake City, UT 48112, USA}
\email{{\tt defernex@math.utah.edu}}

\author{Yu-Chao Tu}

\address{Department of Mathematics, University of Utah, Salt Lake City, UT 48112, USA}
\email{{\tt tu@math.utah.edu}}


\thanks{The research of the first author was partially supported by NSF Grant DMS-1402907 and
NSF FRG Grant DMS-1265285.
The research of the second author was partially supported by NSF FRG Grant DMS-1265285.}

\thanks{Compiled on \today. Filename {\tt \jobname}}

\begin{document}

\begin{abstract}
A theorem of Mumford states that, on complex surfaces, any normal
isolated singularity whose link is diffeomorphic to a sphere is
actually a smooth point. While this property fails in higher
dimensions, McLean asks whether the contact structure that the link
inherits from its embedding in the variety may suffice to characterize
smooth points among normal isolated singularities. He proves that this
is the case in dimension 3. In this paper, we use techniques from birational geometry to
extend McLean's result to a large class of higher dimensional singularities. 
We also introduce a more refined invariant of the link using CR geometry,
and conjecture that this invariant is strong enough to characterize 
smoothness in full generality.
\end{abstract}

\maketitle

\section{Introduction}

Let $X \subset \C^N$ be an $n$-dimensional complex analytic variety
with an isolated singularity at the origin $0 \in \C^N$.
We do not exclude the possibility that $X$ is actually smooth at $0$.

For sufficiently small $ \ep > 0$,
the intersection of $X$ with the sphere of radius $\ep$ centered at $0$
is a differential manifold $L_X$ of real dimension $2n-1$, called the \emph{link}
of the singularity \cite{Mil68}.
The diffeomorphism class of the link is an analytic invariant of the singularity.
If $X$ is smooth at $0$, then the link is diffeomorphic to a sphere, and a theorem of Mumford
states that the converse holds for normal surface singularities.

\begin{thm}[\cite{Mum61}]
A normal isolated surface singularity is smooth if and only if
the link is diffeomorphic to a sphere.
\end{thm}

The converse fails, however, in higher dimensions, even assuming that the
singularity is normal \cite{Br66a,Br66b,Mil68}.

The link inherits a contact structure from its embedding in $X$
given by the hyperplane distribution
\[
\x = T(L_X) \cap J_0(T(L_X)) \subset T(L_X)
\]
where $J_0 \colon T(X) \to T(X)$, with $J_0^2=-1$, is the complex structure.
The contactomorphism class of the link
is an analytic invariant of the singularity \cite{Var80}.
In particular, if $X$ is smooth at $0$ then the link $L_X$ is contactomorphic to the
standard contact sphere $S^{2n-1} \subset \C^n$.
The following question was considered by McLean.

\begin{question}[\cite{McL16}]
\label{q:McL}
Does the contact structure of the link suffice to characterize smooth
points among normal isolated singularities?
\end{question}

McLean proves that this is the case in dimension three.

\begin{thm}[\cite{McL16}]
A normal isolated 3-fold singularity is smooth if and only if
the link is contactomorphic to the standard contact sphere.
\end{thm}

McLean's strategy is to first observe that a normal isolated singularity $0 \in X$
whose link is contactomorphic to the standard contact sphere is numerically $\Q$-Gorenstein.
Via a delicate analysis of the geometry and dynamics of
Reeb vector fields, McLean proves that the Conley--Zehnder index of the
Reeb orbits defined by the various contact forms on the link can be used
to compute discrepancies, hence deducing that $0 \in X$
is actually a $\Q$-Gorenstein singularity with minimal log discrepancy $\mld_0(X) = n$.
When $n=3$, this suffices to conclude that $X$ is smooth at $0$.
In higher dimensions, the same conclusion can be drawn by
assuming a conjecture of Shokurov
stating that if $\mld_0(X) = n$ then $X$ is smooth at $0$
\cite{Sho02}.\footnote{We should remark that this conjecture of Shokurov is a particular case of
a more precise conjecture that has received a lot of attention in birational geometry
due to its relevance in the minimal model program, specifically in connection to the
conjecture on termination of flips.}

In this paper, we show that this approach leads to a positive answer
to the question for a large class of singularities of all dimensions.

\begin{thm}
\label{t:link}
Let $X \subset \C^N$ be a complex analytic variety with a normal isolated singularity at $0$.
Assume that the exceptional divisor of the normalized blow-up of $X$ at $0$
has a generically reduced irreducible component.
Then $X$ is smooth at $0$ if and only if the link $L_X$ of $X$ at $0$ is
contactomorphic to the standard contact sphere.
\end{thm}

The condition on the exceptional divisor of the normalized blow-up of $X$ at $0$
is satisfied, for instance, if
the tangent cone of $X$ at $0$ is reduced
at the generic point of one of its irreducible components.

We deduce the above theorem from McLean's work and the next result, which
brings evidence to Shokurov's conjecture and therefore is of independent interest.

\begin{thm}
\label{t:mld}
Let $X$ be a normal $\Q$-Gorenstein variety of dimension $n$ and $x\in X$ a closed point such that
the exceptional divisor of the normalized blow-up of $X$ at $x$
has a generically reduced irreducible component. Then
$\mld_x(X) \le n$, and equality holds if and only if $X$ is smooth at $x$.
\end{thm}

The proof of this theorem relies on a theorem of Ishii on minimal Mather log discrepancies
\cite{Ish13}. The bound $\mld_x(X) \le n$ in the setting of the theorem is a direct application of
Ishii's result, and our contribution is to observe that equality holds
only if $X$ is smooth at $x$, a property we deduce by looking at the Nash
blow-up of $X$.

We do not know how to extend these results to the case where the
exceptional divisor of the normalized blow-up
has no reduced irreducible components.
We suspect, however, that by keeping track of additional structure of the link
one can still detect smoothness among normal isolated singularities.

To this end, we look at the CR structure that the link inherits from
the complex structure of $X$.
This structure is defined by
\[
T_{1,0}(L_X) = T^{1,0}(X) \cap (T(L_X) \otimes \C).
\]
It is determined by the tangential Cauchy--Riemann equations in $X$ along $L_X$.
The contact structure $\x$ can be recovered as the real part of the
complex bundle $T_{1,0}(L_X) \oplus \ov{T_{1,0}(L_X)}$,
where the bar denotes complex conjugation.

The CR structure of the link uniquely characterizes the singularity \cite{Sch86}.
It has long been used to investigate isolated singularities;
see \cite{Hua06} for a survey on recent studies in this direction.
In his foundational work \cite{Tan75}, Tanaka devoted a chapter to this problem, and asked
whether the sphere $S^{2n-1}$ equipped with any strongly pseudoconvex CR structure that extends
the standard contact structure can be CR embedded into $\C^n$ (cf. \cite[Page~80]{Tan75}).
It is interesting to observe that a positive answer to Tanaka's question would imply
a positive answer to Question~\ref{q:McL} (cf.\ Remark~\ref{r:Tanaka}).

It should be noted, however, that the CR structure is not an invariant
of the singularity, as it may depends on the embedding of $X$ in $\C^N$ and
the radius $\ep$ of the sphere cutting out the link.

In order to define an invariant of the singularity, we
consider an equivalence relation among compact strongly pseudoconvex CR structures
where certain types of deformations of the structure are allowed.
We call such equivalence relation \emph{cohomological CR deformation equivalence}.

Essentially, two compact strongly pseudoconvex CR manifolds
are in the same equivalence class if they can be
deformed into each other using finitely many deformation families, parameterized over $\C$,
that preserve certain cohomological 
dimensions throughout the deformation and whose
total spaces are strongly pseudoconvex CR manifolds (see Definition~\ref{d:CR-equiv}).
The fact that the cohomological CR deformation equivalence class of a link
gives an analytic invariant of the singularity is proved in Theorem~\ref{t:CR}.

Using this notion, we conjecture that a normal isolated singularity
is smooth if and only if the cohomological CR deformation equivalence class
of the link is the class of the standard CR sphere (see Conjecture~\ref{c:CR}). 
This conjecture is closely related to the complex Plateau problem, which 
was completely solved for isolated hypersurface singularities in \cite{Yau81,LY07,DY12}
and has been further investigated for arbitrary singularities in \cite{DGY16}. 

Another way of enhancing the contact structure of the link
is by looking at contact forms.
We consider the contact form defined on the link by the 1-form
\[
\theta = \frac{\sqrt{-1}}{|z|^2} \sum_{i=1}^N (z_i d\ov{z}_i - \ov{z}_i dz_i)
\]
where $(z_i)$ are the coordinates of $z$ in $\C^N$.
The contact structure $\x$ on $L_X$ is simply the kernel of the form $\theta|_{L_X}$.

We expect that smoothness can also be characterized by tracing this form as one
lets the radius of the sphere cutting out the link tend to zero (see Conjecture~\ref{c:contact}).
This conjecture is independent of the previous conjecture
and provides a different perspective on the problem.
We refer to the last section of the paper for more details on this
alternative approach.

\subsection{Acknowledgments}

Our interest in this problem originated from
an inspiring talk Mark McLean gave at the University of Utah in December 2014 on his work \cite{McL16};
we warmly thank him for explaining to us some of the key ideas and techniques behind his proof,
and for his valuable feedback of preliminary drafts of this paper.
We thank Mihai P\u aun for his remarks on an old version of the paper,
Eduard Looijenga for suggesting the proof of Corollary~\ref{c:link-stability}
and Hugo Rossi the proof of Corollary~\ref{c:unique-filling}, and
Steven Yau for pointing out the connection between Conjecture~\ref{d:CR-equiv} and 
his work on the complex Plateau problem.
Finally, we thank the referees for useful comments.

\section{Proof of Theorem~\ref{t:mld}}

Throughout this and the next section, we denote by $K_X$ a canonical divisor of a normal variety $X$.
If $f \colon Y \to X$ is a proper birational map from another normal variety, then
the canonical divisors of $X$ and $Y$ will always be implicitly assumed to be chosen compatibly,
so that $f_*K_Y = K_X$.

We start by recalling some definitions.
Let $X$ be an $n$-dimensional normal variety. We say that $X$ is \emph{$\Q$-Gorenstein} if
the canonical divisor $K_X$ of $X$ is $\Q$-Cartier.

Let $E$ be a prime divisor on a resolution of singularities
$f \colon Y \to X$. If $X$ is $\Q$-Gorenstein, then the \emph{log discrepancy} of $E$ over $X$
is defined by
\[
a_E(X) := \ord_E(K_{Y/X}) + 1
\]
where $K_{Y/X} = K_Y - f^*K_X$ is the relative canonical divisor
(note that this is a $\Q$-divisor).
In general (i.e., without assuming that $X$ is $\Q$-Gorenstein), one defines the
\emph{Mather log discrepancy} of $E$ over $X$ to be
\[
\^a_E(X) := \ord_E(\Jac_f) + 1
\]
where $\Jac_f = \Fitt^0(\Om_{Y/X})$.
By taking the infimum of these numbers
over all choices of $E \subset Y \to X$ such that the image of $E$ in $X$
is a fixed closed point $x \in X$, one defines
the \emph{minimal log discrepancy} $\mld_x(X)$ and
the \emph{minimal Mather log discrepancy} $\^\mld_x(X)$ of $X$ at $x$.

If $X$ is smooth, then log discrepancies and Mather log discrepancies are the same.
In general, on a $\Q$-Gorenstein variety $X$, they compare as follows.
Let $r$ be a positive integer such that $rK_X$ is Cartier, and consider the natural map
\[
(\wedge^n\Om_X)^{\otimes r} \to \O_X(rK_X).
\]
The image of this map is equal to $\fn_{r,X} \. \O_X(rK_X)$ for some ideal
sheaf $\fn_{r,X} \subset \O_X$, which we call the \emph{Nash ideal
of level $r$} of $X$. 
It follows that
\[
\^a_E(X) = a_E(X) + \tfrac 1r \ord_E(\fn_{r,X}).
\]

The name given to $\fn_{r,X}$ comes from the following property, which is certainly well-known
to the experts.
Recall that the \emph{Nash blow-up} $\n \colon \^X \to X$ of an $n$-dimensional variety $X$ is defined
by taking the closure in the Grassmannian bundle $\Gr(\Om_X,n)$ of $\Gr(\Om_{X_\reg},n)$, which
is naturally isomorphic to $X_\reg$.\footnote{Assuming that
$X$ is a closed subvariety of a smooth variety $M$,
the Nash blow-up of $X$ can equivalently be defined by taking the closure of $X_\reg$
in the Grassmannian bundle $\Gr(\Om_M,n)$.
We should point out that the construction of the Nash blow-up
was already considered by Semple \cite{Sem54}, before Nash.}

\begin{prop}
\label{p:Nash}
With the above notation (i.e., assuming that $rK_X$ is Cartier),
the Nash blow-up $\^X$ of $X$ is isomorphic to the blow-up of $\fn_{r,x}$.
\end{prop}

\begin{proof}
Using the Pl\"ucker embedding, it is easy to see that $\^X$ is isomorphic to the
closure of $X_\reg$ in $\P(\wedge^n\Om_X)$ via the natural isomorphism
$X_\reg \cong \P(\wedge^n\Om_{X_\reg})$.
Then the Nash blow-up can be viewed as a blow-up of an ideal as follows.
If $\wedge^n \Om_X \to \L$ is a generically injective map to an
invertible sheaf $\L$ and $\fa \subset \O_X$ is the ideal sheaf
such that the image is equal to $\fa\.\L$, then
$\^X$ is isomorphic to the blow-up of $\fa$. These general
facts about the Nash blow-up are well explained, for instance, in \cite{OZ91}.

There is a generically injective map $(\wedge^n\Om_X)^{\otimes r} \to \O_X(rK_X)$.
Pick $\L$ (and $\fa$) as above so that the tensor product map
$(\wedge^n\Om_X)^{\otimes r} \to \L^{\otimes r}$ factors through $\O_X(rK_X)$,\footnote{For instance,
if $X$ is a closed subvariety of a smooth variety $M$ and
$V \subset M$ is the complete intersection
defined by the vanishing of $c = \codim_MX$ general elements of the ideal of $X$ in $\O(M)$,
then one can take $\L = \om_V|_X$, in which case $\fa = \Jac_V|_X$ where $\Jac_V = \Fitt^n(\Om_V)$
is the Jacobian ideal of $V$.}
giving
\[
(\wedge^n\Om_X)^{\otimes r} \to \O_X(rK_X) \to \L^{\otimes r}.
\]
Note that the image of $(\wedge^n\Om_X)^{\otimes r} \to \L^{\otimes r}$
is equal to $\fa^r\.\L^{\otimes r}$, and since
both $\O_X(rK_X)$ and $\L^{\otimes r}$ are invertible sheaves,
the map $\O_X(rK_X) \to \L^{\otimes r}$ is given by multiplication of a nonzero
element $h \in \O(X)$. It follows that
\[
\fa^r = (h)\.\fn_{r,X}
\]
in $\O_X$, and this implies that the two blow-ups are isomorphic.
\end{proof}

Let us now turn to the setting of the theorem, so that $X$ is $\Q$-Gorenstein, $x \in X$
is a closed point, and the exceptional divisor of the normalized blow-up of $X$ at $0$
has a generically reduced irreducible component.
By \cite[Theorem~1.1]{Ish13},\footnote{The cited theorem, as stated in the published version,
does not apply to our situation.
However, the theorem is stated incorrectly, and the corrected statement does apply.
We have been informed by the author of \cite{Ish13}
that an Erratum is in the process of being submitted.}
the condition on the exceptional divisor of the normalized blow-up implies that
\[
\^\mld_x(X) = n.
\]
Since $\ord_E(\fn_{r,X}) \ge 0$, this immediately gives
\[
\mld_x(X) \le n.
\]
Suppose that $\mld_x(X) = n$. Since $\^a_E(X) - a_E(X) = \frac 1r \ord_E(\fn_{r,X})\in \frac 1r \Z$
for every prime divisor $E$ over $X$,
we must have $\ord_E(\fn_{r,X}) = 0$ for some $E$ with center in $x \in X$,
which is only possible if $\fn_{r,X}$ is locally trivial in a neighborhood of $x$.
Then Proposition~\ref{p:Nash} implies
that the Nash blow-up $\n \colon \^X \to X$ is an isomorphism near $x$.
As we are in characteristic zero, we deduce that $X$ is smooth at $x$ by
\cite[Theorem~2]{Nob75}.

\section{Proof of Theorem~\ref{t:link}}

Theorem~\ref{t:link} simply follows by combining Theorem~\ref{t:mld}
with \cite[Corollary~1.4]{McL16}. For the convenience of the reader,
we outline the various steps of the proof, following \cite{McL16}.
This will show how several different ways of looking at a singularity come into play.

Let $X \subset \C^N$ be a variety with a normal isolated singularity at the origin $0 \in \C^N$,
and let $L_X = X \cap S^{2N-1}_\ep$ be the link of $X$
cut out by a sphere of sufficiently small radius $\ep$.
We can assume without loss of generality that $X$ is smooth away from $0$.

The starting point is to observe that if $L_X$ is diffeomorphic to
the sphere $S^{2n-1}$ then $X$ is \emph{numerically $\Q$-Gorenstein} (cf.\ \cite[Lemma~3.3]{McL16}).
This means that for some (equivalently, for any) resolution of singularities $f \colon Y \to X$
there is a $\Q$-divisor $f^*_\num K_X$ on $Y$, called the \emph{numerical pull-back} of $K_X$,
which is characterized by the properties that $f_*f^*_\num K_X = K_X$ and
$f^*_\num K_X \. C = 0$ for every curve $C \subset Y$ that is contracted by $f$.
We equivalently say that $K_X$ is \emph{numerically $\Q$-Cartier}.\footnote{The proof
actually shows that in our setting $K_X$ is numerically Cartier, which means that
$f^*_\num K_X$ is an integral divisor.}

To see this, let $f \colon Y \to X$ be a resolution such that
$\Ex(f) = \Supp(f^{-1}(0)) = \bigcup E_i$, where $E_i$
are prime divisors. Let $U = f^{-1}(X \cap B^{2N}_\ep)$ where $B^{2N}_\ep \subset \C^N$
is the closed ball of radius $\ep$. Note that $U$ is a $2n$-dimensional orientable real
manifold with boundary $\de U \simeq L_X$.
Since $L_X \simeq S^{2n-1}$, we have $H^2(\de U;\Z) = 0$, and hence the
the map $H^2(U,\de U;\Z) \to H^2(U;\Z)$ is an isomorphism. On the other hand,
by Lefschetz duality the cohomology group $H^2(U,\de U;\Z)$ is isomorphic to $H_{2n-2}(U;\Z)$,
and the latter is generated by the classes $[E_i]$ since $\Ex(f) \inj U$ is a homotopy equivalence.
If $\sum a_i[E_i] \in H_{2n-2}(U;\Z)$ is the element corresponding to
$c_1(T^*(Y)|_U) \in H^2(U;\Z)$ under these isomorphisms, then for every
curve $C \subset \Ex(f)$ we have
\[
\big(\sum a_i E_i\big)\.C = c_1(T^*(Y)|_U)(C) = K_Y \. C.
\]
This means that $f_\num^*K_X := K_Y - \sum a_i E_i$ is the numerical pull-back of $K_X$, and hence
$X$ is numerically $\Q$-Gorenstein.

Using the numerical pullback of $K_X$, one defines the relative
canonical divisor of $X$ by setting
\[
K_{Y/X}^\num :=  K_Y - f_\num^*K_X = \sum a_i E_i.
\]
Then, for every prime divisor $E_i \subset Y$, one defines the
\emph{log discrepancy}
\[
a_{E_i}(X) := \ord_{E_i}(K_{Y/X}^\num) + 1 = a_i+1.
\]
Notice that this definition agrees with the one given in the previous section
under the additional assumption that $K_X$ is $\Q$-Cartier.

The next step relates the contact geometry of the link to log discrepancy.
This is by far the deepest and most difficult part of the proof, and it
is the core contribution of \cite{McL16}. Before we can state the main result of \cite{McL16},
we need to recall some definitions.

Let $M = (M,\x)$ be a \emph{contact manifold}.
That is, $M$ is an odd dimensional differentiable real manifold and
$\x \subset T(M)$ is a hyperplane distribution locally
defined by $\x = \ker\a$ where $\a$ is a differential 1-form such that the top form
$\a \wedge (d\a)^m$ is nowhere zero, where we set $\dim M = 2m+1$. 
If defined globally, any such form $\a$ induces an orientation on the bundle $T(M)/\x$,
and is called a \emph{contact form} on $M$.
If a contact form $\a_0$ on $M$ is given, then we say that the contact manifold is 
\emph{cooriented}, and any other contact form $\a$ is said to be
\emph{cooriented} if it induces the same orientation on $T(M)/\x$ as $\a_0$.
In this paper, we always focus on cooriented contact manifolds with given contact forms.
A \emph{Reeb vector field} of a contact form $\a$ is the unique vector field $R$ on $M$ such that
$d\a(R, -)= 0$ and $\a(R) = 1$. The \emph{Reeb flow} of $\a$ is the flow of $R$.
A \emph{Reeb orbit} of a contact form $\a$ is a closed orbit $\g \colon \R/L\Z \to M$
(where $L$ is the period) of the Reeb flow, so that $\g'(t) = R(\g(t))$.

In the situation at hand, let $\a_0$ be the contact form on the link $L_X$ induced by
the 1-form $i \sum_{\a=1}^N\(z^\a d\ov{z}^\a - \ov{z}^\a dz^\a\)$ on $\C^N$.
Note that the contact structure $\x \subset T(L_X)$ inherits a complex structure
from the complex structure $J_0$ of $T(X)$ that is compatible with $d\a_0$.
Since $L_X \simeq S^{2n-1}$, we have $H^2(L_X,\Z) = 0$,
and therefore the first Chern class $c_1(\x)$ of $\x$ (which is defined
with respect to such complex structure of $\x$) is zero.
This implies that the complex line bundle $\wedge^{n-1}\x$ is trivial.
Then, for any contact form $\a$ that is cooriented with $\a_0$, and any
Reeb orbit $\g \colon \R/L\Z \to L_X$ of $\a$, one can define the
\emph{Conley--Zehnder index} $\CZ(\g) \in \frac 12\Z$ of $\g$
by looking how many times the linearized return map of $\g$ goes
around the origin of $\C$ in a given trivialization of the dual of $\wedge^{n-1}\x$
(see \cite[Lemma~2.10]{McL16} or the original source \cite{RS93} for the precise definition).
McLean then defines the \emph{lower SFT index} of $\g$ to be
\[
\lSFT(\g) := \CZ(\g) - \frac 12 \dim \ker(D\phi_L|_{\xi\cap T_{\gamma(0)}L_X} - \id) + (n-3),
\]
where $D\f_L|_{T_{\gamma(0)}L_X} \colon T_{\gamma(0)}L_X \to T_{\gamma(0)}L_X$
is the linearized return map of $\g$.
For any contact form $\a$ with $\ker\a = \x$, we define the \emph{minimal SFT index} of
$\a$ to be $\mi(\a) = \inf_\g\lSFT(\g)$, where the minimum is taken over all Reeb orbits $\g$
of $\a$, and the \emph{highest minimal SFT index} of $(M,\x)$
by $\hmi(M,\x) := \sup_\a \mi(\a)$ where the supremum is taken over all non-zero
1-forms $\a$ with $\ker(\a) = \x$ respecting the coorientation of $\x$ induced by $\a_0$.

A good property of $\hmi(M,\x)$ is its invariance up to coorientation preserving contactomorphism.
The following is the main result in McLean's paper.

\begin{thm}[\protect{\cite[Theorem~1.1]{McL16}}]
\label{t:McL}
Let $X \subset \C^N$ be a variety with a normal isolated singularity at $0$,
and assume that $X$ is numerically $\Q$-Gorenstein and $H^2(L_X; \Q) = 0$.
\begin{enumerate}
  \item If $\mld_0(X) \geq 1$, then $\hmi(L_X,\x) = 2\mld_0(X) - 2$.
  \item If $\mld_0(X) < 1$, then $\hmi(L_X,\x) < 0$.
\end{enumerate}
\end{thm}

Granting this result, we can now finish the proof of Theorem~\ref{t:link}.
Let $X \subset \C^N$ be a complex analytic variety with a normal isolated singularity at $0$.
If $X$ is smooth at $0$, then its link $L_X$ of $X$ at $0$ is clearly contactomorphic to
the standard sphere. We need to check the other direction.

Assume that $L_X \simeq S^{2n-1}$. Then, as we saw above, $X$ is numerically $\Q$-Gorenstein.
By applying Theorem~\ref{t:McL}, we see that the minimal log discrepancy $\mld_0(X)$
only depends on the contact structure of the link, and therefore
is the same as $\mld_0(\C^n)$, which is equal to $n$.
Recall that, at this point, we only know that $X$ is numerically $\Q$-Gorenstein.
It follows, however, by \cite[Corollary~1.4]{BdFFU15} that $X$ is actually
$\Q$-Gorenstein.\footnote{This corollary combines several logical steps,
which might be worthwhile to explain. First, one observes that
the fact that the minimal log discrepancy defined numerically is positive
implies that $X$ is log terminal in the sense defined in \cite{dFH09},
and therefore there is an effective $\Q$-divisor $\D$ on $X$ such that
$K_X + \D$ is $\Q$-Cartier and $(X,\D)$ is log terminal. It
follows then by a well-known extension to log pairs of Elkik's theorem \cite{Elk81}
that $X$ has rational singularities, and
one concludes that $K_X$ is $\Q$-Cartier by \cite[Theorem~5.11]{BdFFU15}.}
Since we have assumed that the exceptional divisor of
the normalized blow-up of $X$ at $0$ has a generically reduced component,
we conclude by Theorem~\ref{t:mld} that $X$ is smooth at $0$.

\section{CR geometry}

This section is devoted to recalling some general facts about CR manifolds
(e.g., see \cite{DT06} for a general reference).
All manifolds will be implicitly assumed to be $C^\infty$, connected, and oriented.

A CR structure on a manifold $M$ with real dimension $2n-1$ is an
$(n-1)$-dimensional subbundle $T_{1,0}(M)$
of the complexified tangent bundle $T(M)\otimes\C$ that is closed under Lie bracket and satisfies
$T_{1,0}(M)\cap \ov{T_{1,0}(M)} = {0}$.
The \emph{Levi distribution} of a CR manifold $(M,T_{1,0}(M))$ is the subbundle
$H(M) = \Re(T_{1,0}(M) \oplus \ov{T_{1,0}(M)})$ of $T(M)$.
A \emph{pseudo-Hermitian structure} on $M$ consists of a global differential 1-form $\theta$
on $M$ such that $H(M) = \ker\theta$.
A CR manifold $(M,T_{1,0}(M))$ is said to be \emph{strongly pseudoconvex}
if for some pseudo-Hermitian structure $\theta$ the Levi form $-i\,d\theta$ is positive definite, 
that is, $-i\,d\theta(Z,\ov Z) > 0$ for every nonvanishing local section $Z$ of $T_{1,0}(M)$.

Any strongly pseudoconvex CR manifold $(M,T_{1,0}(M))$ determines a
contact manifold $(M,\x)$ by setting $\x = H(M)$, but the same contact manifold
may arise from different CR structures.\footnote{As we will
discuss below, this may occur, for instance, when taking different links of the same
isolated singularity.}

By keeping track of a strongly pseudoconvex CR structure of a contact manifold $M$,
we obtain a rather more rigid picture in terms of Stein fillings.
To this end, it is useful to recall
the following two fundamental theorems concerning embeddability and fillings of
strongly pseudoconvex CR manifolds.

\begin{thm}[\cite{BdM75}]
\label{t:BdM}
Any compact strongly pseudoconvex CR manifold $M$ of dimension at least 5
is CR embeddable in $\C^N$ for some $N$.
\end{thm}

\begin{thm}[\cite{HL75}]
\label{t:HL}
Any compact strongly pseudoconvex CR submanifold $M \subset \C^N$ is the boundary of
a unique Stein space $V \subset \C^N$ with only normal isolated singularities.
\end{thm}

It follows from these theorems that any compact connected strongly pseudoconvex
CR manifold $M$ of dimension at least 5 is isomorphic to the boundary of a Stein
space $V$ with only normal isolated singularities.
While these results do not quite imply immediately that the filling is unique up to isomorphism
(in principle the Stein space $V$
may depend on the embedding $M \subset \C^N$), it is a general fact that that is the case. 
In fact, the above theorems, together with results from \cite{Ros64,RT77}, yield the following property.

\begin{cor}
\label{c:unique-filling}
Up to isomorphism, a compact strongly pseudoconvex CR 
manifold $M$ of dimension at least $5$ is the boundary of
a unique Stein space $V$ with only normal isolated singularities.
\end{cor}

\begin{proof}
By Theorem~\ref{t:BdM}, $M$ can be embedded in an affine space $\C^N$, and
Theorem~\ref{t:HL} then implies that, within its embedding, it can be filled
by a Stein space $V$ with normal isolated singularities, so
that $M$ is CR isomorphic to the boundary $\de V$ of $V$. 
Let $\ov V = V \cup \de V$ denote the closure of $V$ in $\C^N$. 

Unicity of such filling follows from the results of \cite{Ros64,RT77}
which provide a canonical way of constructing $V$ up to isomorphism.
Specifically, let $A(M)$ be the Banach algebra of continuous functions on $M$ 
that can be approximated (in the uniform norm) by 
$C^\infty$ functions on $M$ satisfying the tangential Cauchy--Riemann equations.
Let $S(M)$ be the Gelfand spectrum of $A(M)$; this is the set
of maximal ideals of the ring of homomorphisms from $A(M)$ to $\C$.
By \cite[Theorem~6]{Ros64}, reinterpreted in the intrinsic setting treated in \cite{RT77}, 
the spectrum $S(M)$ can be given a complex structure so that
$S(M) = S \cup \de S$ where $S$ is a Stein space with CR boundary $\de S \cong M$, and
it follows by the definition of Gelfand spectrum that $S$ is normal. 
Moreover, for every Stein filling $V$ as above, there is a neighborhood $U \subset \ov V$ 
of $\de V$ and an injective map $\Psi \colon U \to S(M)$ mapping $\de V$ CR isomorphically to $\de S$
and inducing a holomorphic map $U \cap V \to S$. 

Arguing as in the proof of \cite[Proposition~1.1]{Yau11}, we see that $\Psi$
extends to a finite map $\Phi \colon \ov V \to S(M)$ which satisfies
$\Phi^{-1}(\de S) =\de V$, and induces a proper holomorphic map $V \to S$.
Since $\Phi$ is injective on $U$, it follows that $\Phi$ has degree one, and therefore
it is an isomorphism because $S$ is normal. 
\end{proof}

\begin{rmk}
\label{r:Tanaka}
Since every $2n-1$ dimensional strongly pseudoconvex CR submanifold $M \subset \C^n$
is the boundary of a Stein domain $U \subset \C^n$,
it follows from Corollary~\ref{c:unique-filling}
that a positive answer to the question of Tanaka mentioned in the introduction
would imply a positive answer to Question~\ref{q:McL}.
\end{rmk}

Before stating Theorem~\ref{t:HLY} below, we need to recall two notions. 
The first is the notion of Kohn--Rossi cohomology.
For the precise definition, we refer to \cite{KR65,FK72,Tan75}
(see also, \cite[Section~3]{Yau81} for a discussion). 
In short, the Kohn--Rossi cohomology
groups $H_\KR^{p,q}(M)$ of a strongly pseudoconvex CR manifold $M = (M, T_{1,0}(M))$
of dimension $2n-1$ are given, for a fixed $p$, by the cohomology of the boundary complex
\[
0 \to \cB^{p,0} \xrightarrow{\ov\de_b} \cB^{p,1} \xrightarrow{\ov\de_b} \dots
\xrightarrow{\ov\de_b}\cB^{p,n-1} \to 0
\]
where $\cB^{p,q}$ is the space of $(p,q)$-forms on $M$ 
satisfying the tangential Cauchy--Riemann equations, 
and $\ov\de_b \colon \cB^{p,q} \to \cB^{p,q+1}$ 
is the composition of the regular de Rham differential operator
followed by a projection from the space of $(p,q+1)$-forms onto $\cB^{p,q+1}$.

The following theorem implies in particular that most of the Kohn--Rossi cohomology
groups of a link of an isolated singularity are analytic invariants of the singularity, 
a fact that will be used later.

\begin{thm}[\cite{Yau81}]
\label{t:Yau81}
If $M$ is the boundary of a Stein space $V$ of complex dimension $n \ge 3$
with only isolated singularities $x_1,\dots,x_m$, 
then 
\[
\dim H_\KR^{p,q}(M) = \sum_{i=1}^m b_{x_i}^{p,q+1}  \for 1 \le q \le n-2, 
\]
where $b_{x_i}^{p,q+1} := \dim H^{q+1}_{[x_i]}(V,\Om_V^p)$ is the 
\emph{Brieskorn invariant} of type $(p,q+1)$ at $x_i$, 
which is a local analytic invariant of the singularity $x_i$.
\end{thm}

We also recall the definition of $p$-normal, which appears in the next theorem. 
It will not be used in the rest of the paper. 

\begin{defi}
A complex analytic variety $X$ is said to be 
\emph{$p$-normal} if $\O_X^{[p]} = \O_X$, where $\O_X^{[p]}$ is the sheaf 
defined by the pre-sheaf given by $U \mapsto \dirlim_{Z \subset U} \G(U\setminus Z, \O_X)$,
with $Z$ ranging among all closed subvarieties of $U$ of dimension $\le p$
and the limit directed by inclusion.
\end{defi}

\begin{thm}[\cite{HLY06}]
\label{t:HLY}
Let $\M$ be a real manifold of dimension $2n+1$, with $n \ge 3$, and let
$f \colon \M \to \D$ be a proper $C^\infty$ submersion, where $\D = \{t \in \C \mid |t| < 1\}$
is the open unit interval, such that every fiber $\M_t = f^{-1}(t)$,
for $t \in \D$, is a connected $(2n-1)$-dimensional submanifold of $\M$.
Suppose that $T_{1,0}(\M)$ is a strongly pseudoconvex CR structure on $\M$
such that $T_{1,0}(\M_t) := T_{1,0}(\M) \cap (T(\M_t)\otimes\C)$ is a
strongly pseudoconvex CR structure on $\M_t$ for every $t$.

Then there exist a unique (up to isomorphism) 2-normal
Stein space $\U$ which has $\M$ as part of its smooth strongly pseudoconvex boundary, and
a holomorphic map $g \colon \U \to \D$
such that, for every $t$, the fiber $\U_t = g^{-1}(t)$
is a Stein space with $\M_t$ as its smooth strongly pseudoconvex boundary.

Moreover, if the dimension of the Kohn--Rossi cohomology group $H_\KR^{0,1}(\M_t)$ of the fibers of $f$
is constant as a function of $t \in \D$, then every fiber $\U_t$ has only normal
isolated singularities.
\end{thm}

\begin{rmk}
The first assertion of Theorem~\ref{t:HLY} is a restatement of the first part of
\cite[Main Theorem]{HLY06} including statement (I). Regarding the second assertion,
the fact that $\U_t$ has normal singularities
follows by applying \cite[Corollary~1.5 and Remark~1.7]{HLY06} for all $\ep_0 \in (0,1)$, 
and the fact that the singularities are isolated follows from 
the proof of \cite[Corollary~1.5]{HLY06}.
\end{rmk}

\section{A CR theoretic invariant of singularity}

Let us now focus on the case of a link of a normal isolated singularity.
The goal is to refine the structure of the link by taking
into account the CR structure inherited from the embedding in the complex variety.
By allowing a certain type of deformation of such a structure,
we introduce a finer invariant of the singularity which we
conjecture is sufficient to characterize smooth points among normal isolated singularities.

For reasons that will be clear in the discussion that follows, we
need to be more precise in the definition of link and
consider complex analytic varieties $X$ that are only locally closed
in $\C^N$ (instead of closed in $\C^N$).
For the reminder of this section,
given a locally closed complex analytic variety $X \subset \C^N$,
we simply say that $X$ is a \emph{variety} in $\C^N$.

Let $X \subset \C^N$ be an $n$-dimensional variety
with a normal isolated singularity at the origin $0$.
Consider the function $\r \colon X \cup \de X \to \R$ given by $z \mapsto |z|$.

\begin{defi}
We say that $\a \in (0,\infty)$ is a \emph{critical value} of $\r$ if
$\a = \r(z)$ where either $z \in X_\reg$ is a point where $\r$ is not a submersion,
or $z \in (\Sing X) \cup \de X$. We set
\[
\a^*(X) := \sup \{ \a > 0 \mid \text{ $\a$ is not a critical value of $\r$ } \} \in (0,\infty].
\]
\end{defi}

For every $\ep \in (0,\a^*(X))$, the set
\[
L_{X,\ep} := S^{2N-1}_\ep \cap X
\]
is a link of $X$ (cf.\ \cite[Proposition~(2.4)]{Loo84}).

We consider each link $L_{X,\ep}$
with the CR structure $T_{1,0}(L_{X,\ep})$ induced from the embedding of $L_{X,\ep}$ in $X$.
The following result of Scherk can be viewed as a particular case of Corollary~\ref{c:unique-filling}.

\begin{thm}[\protect{\cite[Theorem~4]{Sch86}}]
Let $X \subset \C^N$ and $X' \subset \C^{N'}$ be two varieties
with normal isolated singularities at the respective origins $0 \in \C^N$ and $0' \in \C^{N'}$.
If for some $0 < \ep < \a^*(X)$ and
$0 < \ep' < \a^*(X')$ the links $L_{X,\ep}$ and $L_{X',\ep'}$
are CR isomorphic, then the germs $(X,0)$ and $(X',0')$ are
analytically equivalent.
\end{thm}

This results says that the CR structure of a link of a normal isolated singularity
determines uniquely the singularity. This does not mean that the CR
structure defines an invariant of the singularity.
The CR structure of a link is not an invariant of the singularity
because it may depend on the radius $\ep$ as well as the embedding $X \subset \C^N$.
An explicit example where it depends on the embedding is given in \cite[Page~401]{Sch86}.

In order to define a CR theoretic invariant of singularity, we
introduce a suitable equivalence relation among compact strongly pseudoconvex CR structures.
Along the lines of \cite[Definition~1.1]{HLY06}, we start with the following definition.

\begin{defi}
\label{d:CR-family}
A \emph{CR deformation family} of relative dimension $2n-1$ is a
strongly pseudoconvex CR manifold $\M$ of dimension $2n+1$ with a
proper $C^{\infty}$ submersion
$f\colon \M \rightarrow \D$, where $\D = \{t \in \C \mid |t| < 1\}$
denotes the open unit interval, such that for any $t \in \C$, the fiber
$\M_t = f^{-1}(t)$ is a strongly pseudoconvex CR manifold with CR structure
$T_{1,0}(\M_t) = T_{1,0}(\M) \cap (T(\M_t) \otimes\C)$.
\end{defi}

\begin{rmk}
It follows by Gray's stability theorem that if $f\colon \M \rightarrow \D$
is a CR deformation family, then for any $s,t \in \D$
the fibers $\M_s$ and $\M_t$, equipped with the induced contact structures
$H(\M_s)$ and $H(\M_t)$, are contactomorphic.
\end{rmk}

\begin{defi}
\label{d:cohomologically CR-family}
A CR deformation family $f\colon \M \rightarrow \D$ as in Definition~\ref{d:CR-family}
is said to be \emph{cohomologically rigid} if, additionally,
the dimensions of the Kohn--Rossi cohomology groups $H_\KR^{p,q}(\M_t)$
are constant as functions of $t \in \D$ for all $p$ and $1 \le q \le n-2$.
\end{defi}

\begin{defi}
\label{d:CR-equiv}
We say that two compact strongly pseudoconvex CR manifolds $(M, T_{1,0}(M))$ and
$(M', T_{1,0}(M'))$ of the same dimension $2n-1$ are \emph{cohomologically CR deformation equivalent}
if there exists a finite collection of cohomologically rigid CR deformation families
$f^i\colon \M^i \rightarrow \D$ of relative dimension $2n-1$, 
indexed by $i \in \{1,\dots,k\}$, and pairs of points $s_i,t_i \in \D$, such that
\begin{enumerate}
\item
$(\M^i_{t_i}, T_{1,0}(\M^i_{t_i})) \simeq (\M^{i+1}_{s_{i+1}}, T_{1,0}(\M^{i+1}_{s_{i+1}}))$ 
for $i \in \{1,\dots,k-1\}$;
\item
$(\M^1_{s_1}, T_{1,0}(\M^1_{s_1})) \simeq (M, T_{1,0}(M))$ and 
$(\M^k_{t_k}, T_{1,0}(\M^k_{t_k})) \simeq (M', T_{1,0}(M'))$.
\end{enumerate}
\end{defi}

It is immediate to see that this definition yields an equivalence relation
among compact strongly pseudoconvex CR manifolds.

\begin{thm}
\label{t:CR}
Suppose that $X \subset \C^N$ and $X' \subset \C^{N'}$ are two
varieties with analytically isomorphic germs of normal isolated singularities $(X,0) \cong (X',0')$
at the respective origins $0 \in \C^N$ and $0' \in \C^{N'}$.
Then for any $0 < \ep < \a^*(X)$ and $0 < \ep' < \a^*(X')$
the links $L_{X,\ep}$ and $L_{X',\ep'}$ are cohomologically CR deformation equivalent.
\end{thm}

Before we can prove the theorem, we need a stability property of links of isotrivial
families of singularities, which is stated below in Corollary~\ref{c:link-stability}.
As we were unable to find a reference
in the literature, we include here a proof which we learned from Looijenga.

Let $T \subset \C^m$ be a smooth variety, and let
$\X \subset \C^N \times T$ be a family of varieties $\X_t \subset \C^N$ parameterized by
$T$. We assume that each $\X_t$ has an isolated singularity at $0 \in \C^N$
and is smooth elsewhere.
Let $\a^*(\X_t)$ be defined with respect to this embedding.
We identify $T$ with the section $\{0\} \times T \subset \X$.

\begin{defi}
The family $\X \to T$ is said to satisfy the \emph{link stability property} if
every point $t_0 \in T$ admits an analytic open neighborhood $U \subset T$ such that
$\a^*(\X_t) > \ep$ for some $\ep > 0$ and every $t \in U$.
\end{defi}

\begin{prop}
\label{p:link-stability}
With the above notation, assume that the strata $\X \setminus T$ and $T$ of $\X$
satisfy Whitney's condition~({\it b}) in $\C^N \times \C^m$. 
Then the family $\X \to T$ satisfies the link stability property.
\end{prop}

\begin{proof}
Consider the function $r \colon \X \to \R \times T$ given by $(z,t) \mapsto (|z|,t)$,
where $z$ are the coordinates of $\C^N$ and $t \in T$.
We need to prove that the set of critical points of $r$ in $\X \setminus T$
has no accumulation point in $T$.
Let $(z_i,t_i)$ be a sequence of points in $\X \setminus T$ converging to a
point $(0,t) \in T$. For every $i$, let $\ell_i \subset \C^N \times \C^m$ be the real line
spanned by the vector $z_i$ in $\C^N \times \{t_i\}$,
and let $\t_i \subset \C^N \times \C^m$ be the real tangent space of $\X$ at $(z_i,t_i)$.
Whitney's condition~({\it b}) tells us that
if the sequence of lines $\ell_i$ has limit $\ell$ and the sequence planes $\t_i$ has
limit $\t$, then $\ell \subset \t$. This implies that the restriction of $r$ to
a sufficiently small analytic neighborhood of $(0,t)$ has no critical points in $\X \setminus T$.
\end{proof}

\begin{cor}
\label{c:link-stability}
Let $X$ be a variety with an isolated singularity at a point $x_0$ and smooth elsewhere.
Let $T \subset \C^m$ be a smooth variety, and let $\f \colon X \times T \inj \C^N \times T$ be a locally
closed embedding that respects the projections to $T$ and maps $\{x_0\} \times T$ to $\{0\} \times T$.
Denote by $\X$ the image of $\f$.
Note that each fiber $\X_t$ of $\X \to T$ is a variety in $\C^N$ with an isolated singularity at $0$;
let $\a^*(\X_t)$ be defined with respect to this embedding.
Then $\X \to T$ satisfies the link stability property.
\end{cor}

\begin{proof}
Since $X \setminus \{x_0\}$ is smooth,
$(X \setminus \{x_0\})$ and $\{x_0\}$
satisfy Whitney's condition~({\it b}) with respect to any given embedding $X \subset \C^k$,
and hence $(X \setminus \{x_0\}) \times T$ and $\{x_0\} \times T$
satisfy Whitney's condition~({\it b}) in $\C^k \times \C^m$.
This is an intrinsic property of the variety
(e.g., see \cite[Section~4]{Mat12}), and therefore the strata
$\X \setminus T$ and $T$ satisfy Whitney's condition~({\it b}).
Here, as before, we identify $T$ with the section $\{0\} \times T$. 
Then the property follows by Proposition~\ref{p:link-stability}.
\end{proof}

We now turn to the proof of Theorem~\ref{t:CR}.
The next two lemmas deal with two special cases of the theorem.

\begin{lem}
\label{l:CR-1}
Let $X \subset \C^N$ be a variety with a normal isolated singularity at $0$.
Then for every $0 < \ep' \le \ep < \a^*(X)$ the links $L_{X,\ep}$ and $L_{X,\ep'}$
are cohomologically CR deformation equivalent.
\end{lem}

\begin{proof}
Let $\X = X \times \C \subset \C^{N+1}$ and $h \colon \X \to \C$ be the projection onto the last factor.
Then let $\M = S_\ep^{2N+1} \cap h^{-1}(\D_\ep)$ where $\D_\ep \subset \C$ is the open disk of radius $\ep$,
and let $f \colon \M \to \C$ be the composition of $h|_{\M}$ with an isomorphism 
$g\colon \D_\ep \to \D$. Note that $\M$ is an open
subset of the boundary of a Stein manifold, and in particular it is
a strongly pseudoconvex CR manifold. 
Moreover, we have $\M_s = L_{X,\ep}$ and $\M_t = L_{X,\ep'}$ when 
$s = g(0)$ and $t = g\big(\sqrt{\ep^2 - (\ep')^2}\big)$.
Every fiber $\M_t$ is a link of a smooth point, and hence is strongly pseudoconvex CR manifold,
and the dimensions of the groups $H^{p,q}_\KR(\M_t)$ are constant for $1 \le q \le n-2$ by
Theorem~\ref{t:Yau81}.
Then $f$ is a cohomologically rigid CR deformation family connecting $L_{X,\ep}$ to $L_{X,\ep'}$.
\end{proof}

\begin{lem}
\label{l:CR-2}
Let $X$ be a complex variety with a normal isolated singularity at a point $x_0$.
Let $\f \colon X \times \C \inj \C^N \times \C$ be a locally
closed embedding that respects the projections to $\C$ and maps $\{x_0\} \times \C$ to $\{0\} \times \C$.
Denote by $\X$ the image of $\f$.
Then, for sufficiently small $\ep > 0$, the links $L_{\X_0,\ep}$ and $L_{\X_1,\ep}$
(defined with respect to the corresponding embeddings $\X_t \subset \C^N$)
are cohomologically CR deformation equivalent.
\end{lem}

\begin{proof}
Let $U \subset \C$ be a bounded open set containing the real interval $[0,1]$.
By Corollary~\ref{c:link-stability} and the compactness of the closure of $U$ in $\C$,
there exists an $\ep > 0$ such that $\a^*(\X_t) > \ep$ for every $t \in U$.
We fix real numbers
\[
0 = t_0 < t_1 < \dots < t_{m-1} < t_m = 1
\]
such that $t_i - t_{i-1} \le \ep$ for all $i$.
For every $i$, let
\begin{align*}
\D_i &= \{t \in \C \mid |t-t_i| < \ep\}, \\
S_i &= \{(z,t) \in \C^N\times \C \mid \|(z,t) - (0,t_i)\| = \ep\}.
\end{align*}
Note that $[0,1] \subset \bigcup_{i=0}^m \D_i$.
Without loss of generality, we can assume that $\bigcup_{i=0}^m \D_i \subset U$.
Denoting by $h \colon \X \to \C$ the projection map, we define
\[
\M^i := S_i \cap h^{-1}(\D_i),
\]
and let $f^i \colon \M^i \to \D_i$ be the map induced by $h$.
By construction, $\M^i$ is a strongly pseudoconvex CR manifold. 
For every $i$ and every $t \in \D_i$, the fiber $\M^i_t$ is a link of $\X_t$, 
and hence it is a strongly pseudoconvex CR manifold.
Moreover, the dimensions of the groups $H^{p,q}_\KR(\M_t)$ are constant for $1 \le q \le n-2$ by
Theorem~\ref{t:Yau81}.
After composing with isomorphisms $\D_i \to \D$,
we obtain a finite collection of cohomologically rigid CR deformation families
as in Definition~\ref{d:CR-equiv}, connecting
$L_{\X_0,\ep}$ to $L_{\X_1,\ep}$.
\end{proof}

\begin{proof}[Proof of Theorem~\ref{t:CR}]
After replacing $X$ and $X'$ with analytic open
neighborhoods of the respective origins,
we may assume that $X = X'$ and we are given two locally closed embeddings
$j \colon X \inj \C^N$ and $j' \colon X \inj \C^{N'}$.
For every $t \in \C$, let $j_t \colon X \to \C^N$ denote the composition
of $j$ with the rescaling map $\r_t \colon \C^N \to \C^N$ sending $z \mapsto tz$. Define
$j'_t \colon X' \inj \C^{N'}$ in a similar way.

Consider the locally closed embedding
\[
\f \colon X \times \C \inj \C^N \times \C^{N'} \times \C, \quad \f(x,t) = (j_{1-t}(x),j'_t(x),t).
\]
Let $\X$ denote the image of $\f$, and let
$\X \to \C$ be the projection onto the last factor.
For every $t \in \C$, $\X_t$ is an isomorphic image of $X$ in $\C^N \times \C^{N'}$.
The inclusion $\X_0 \subset \C^N \times \{0'\}$ is naturally identified with $j$, and
the inclusion $\X_1 \subset \{0\} \times \C^{N'}$ with $j'$.
Going back to the notation of the statement of the theorem,
this means that there are natural identifications
$\X_0 = X \subset \C^N$ and $\X_1 = X' \subset \C^{N'}$.
Then the assertion follows from Lemmas~\ref{l:CR-1} and~\ref{l:CR-2}.
\end{proof}

We are interested in the following immediate consequence of Theorem~\ref{t:CR}.

\begin{cor}
The cohomological CR deformation equivalence class of a link
of an isolated singularity of a complex variety $X$
is an analytic invariant of the germ of the singularity.
\end{cor}

If $X = \C^n \subset \C^N$ is a linear subspace, then for every $\ep > 0$
the link $L_{X,\ep}$ is CR isomorphic to $(S^{2n-1}, T_{1,0}^\st)$, where
$T_{1,0}^\st := T^{1,0}(\C^n) \cap (T(S^{2n-1}) \otimes \C)$ is
the \emph{standard CR structure} of $S^{2n-1}$.

If $X \subset \C^N$ is an $n$-dimensional variety passing through the origin
that is smooth but not linear at $0$, then a link
$L_{X,\ep}$ of $X$ at $0$ may not be CR isomorphic to the standard CR sphere.
Nonetheless, it is cohomologically CR deformation equivalent to it.

We expect that this property characterizes smoothness among normal isolated singularities.

\begin{conj}
\label{c:CR}
Let $X \subset \C^N$ be a variety with a normal isolated singularity at $0$.
Then $X$ is smooth at $0$ if and only if
the cohomological CR deformation equivalence class of the link is the class of the
standard CR sphere $(S^{2n-1},T_{1,0}^\st)$.
\end{conj}

Suppose, by way of contradiction, that $X \subset \C^N$ has a (non-regular) normal isolated singularity
at the origin $0$ and yet its link is cohomologically CR deformation equivalent to the standard CR sphere. 
We can assume that $X$ has  dimension $n \ge 3$. 
Let $f^i \colon \M^i \to \D$, for $i = 1,\dots,k$, 
be the sequence of cohomologically rigid CR deformation families as in Definition~\ref{d:CR-family}, 
connecting a link of $0 \in X$ to a link of a smooth point. 

By Theorem~\ref{t:HLY}, for every $i$ there exists a fiberwise filling of $\M^i$. 
That is, there exists a Stein space $\U^i$ which has $\M^i$ as part of its smooth boundary, 
and an holomophic map $g^i\colon \U^i \to \D$ such that, for every $t \in \D$, 
the fiber $\U^i_t$ is a Stein space with $\M^i_t$ as its smooth boundary. 
Moreover, each fiber $\U^i_t$ has isolated normal singularities. 
It follows by Corollary~\ref{c:unique-filling} that 
one of the fibers of $\U^1 \to \D$ is isomorphic to an open neighborhood of $0$ in $X$,
and $\U^k \to \D$ has a smooth and contractible fiber.

We trace through these families. By applying 
Ehresmann's fibration theorem in the context of families of manifolds with boundary
(when moving across a family) and
Corollary~\ref{c:unique-filling} to move from one family to the next,
we end up with a family $\U^i \to \D$ (for some index $i$) 
contains both a singular fiber, say $\U^i_0$, and a smooth and contractible one. 
After possibly shrinking the base, we can assume without loss of generality
that every fiber $\U^i_t$, for $t \ne 0$, is smooth and contractible, while the 
central fiber $\U^i_0$ has normal isolated singularities.
 
If one drops the condition that the central fiber is normal, 
then it is easy to create examples like this (e.g., see Example~\ref{eg} below),
but we do not know any example where the central fiber is normal. 
The fact that the contact structure on the boundaries of the fibers $\U^i_t$
is constant in the family should put further constraint, 
possibly leading to a contradiction and hence a proof of the conjecture. 

\begin{eg}
\label{eg}
Let $Q \subset \P^{n+1}$ be a smooth quadric. Let $x \in Q$ be a point and 
$L \subset \P^n$ a tangent line to $Q$ at $x$ that is not contained in $Q$. 
Then let $U \subset Q$ be the open set cut out by a small open ball in $\P^{n+1}$ centered at $x$. 
We regard $\P^{n+1} \times \P^1$ as embedded in a projective space $\P^N$ 
by the complete linear series of
$\O(1,1)$. Let $\p \colon \P^N \rat \P^{N-1}$ be the linear projection 
from a point $y \in (L \setminus \{x\}) \times \{0\}$, where $0 \in \P^1$ is some fixed point. 
It is easy to see that the induced map $Q \times \P^1 \to \P^{N-1}$ is an isomorphism 
onto its image away from the point $(x,0)$. 
Then let 
\[
\U = \p(U \times \C) \subset \P^{N-1},
\]
where $0 \in \C \subset \P^1$. There is a natural holomorphic map $\U \to \C$
induced by the projection $U \times \C \to \C$. 
The fiber $\U_0$ has a non-normal isolated singularity (the image of $(x,0)$)
and all other fibers $\U_t$ are smooth and contractible. 
\end{eg}

\section{An approach via contact forms}

In this last section, we propose an alternative point of view on the problem.
The idea is to look at the contact form induced on the
link from a certain 1-form on $\C^N$ and trace how it varies as we shrink the radius of the link.
We consider the 1-form
\[
\theta = \frac{\sqrt{-1}}{|z|^2} \sum_{i=1}^N (z_i d\ov{z}_i - \ov{z}_i dz_i).
\]
The reason for the normalization (the denominator $|z|^2$)
is that it makes the form invariant under the $\C^*$-action $z \mapsto \l z$.

If $X=\C^n \subset \C^N$ is a linear subspace, then for every $\ep$
the link $(L_{X,\ep}, \theta_{X,\ep})$ is strictly contactomorphic to $(S^{2n-1}, \theta_\st)$, where
$S^{2n-1}$ is the unit sphere and $\theta_\st$ is the restriction of $\theta$.
We call $\theta_\st$ the \emph{standard contact form} of $S^{2n-1}$.
If $X \subset \C^N$ is smooth but not linear, then in general
$(L_{X,\ep}, \theta_{X,\ep})$ is not strictly contactomorphic to $(S^{2n-1}, \theta_\st)$.
Nevertheless, we have the following property.

\begin{prop}
\label{p:contact}
If $X \subset \C^N$ is smooth, when, after composing with suitable
diffeomorphisms $L_{X,\ep} \simeq S^{2n-1}$, the contact forms
$\theta_{X,\ep}$ specialize to $\theta_\st$ as $\ep \to 0$.
\end{prop}

\begin{proof}
Instead of shrinking $\ep$, we obtain the same effect if we keep the radius fixed
and instead deform $X$ to its tangent space at $0$. To be precise, suppose
that $X = \{f_i(z)=0\} \subset \C^N$. Using the $\C^*$-action $z \mapsto \l z$ on $\C^N$,
we define a flat family of varieties $X_\l = \{f_i(\l z)=0\}$ with $X_\l \cong X$
for all $t \ne 0$ and $X_0 = T_0(X) \cong \C^n$.
As $\theta$ is invariant under this $\C^*$-action, we have
\[
(L_{X,\ep},\theta_{X,\ep}) \simeq (L_{X_\ep,1},\theta_{X_\ep,1})
\]
for every small $\ep > 0$.
The advantage of deforming the equations of $X$ rather than shrinking the radius is that
we take the limit of the right-hand side of the above isomorphism, and get
\[
\lim_{\ep \to 0} (L_{X_\ep,1},\theta_{X_\ep,1}) =
(L_{X_0,1},\theta_{X_0,1}) = (S^{2n-1},\theta_\st).
\]
\end{proof}

We expect that this property characterizes smoothness among varieties with
a normal isolated singularity.

\begin{conj}
\label{c:contact}
Assume that $X \subset \C^N$ is a variety with
a normal isolated singularity at $0$ and that $L_{X,\ep} \simeq S^{2n-1}$ and
the contact forms $\theta_{X,\ep}$ specialize to $\theta_\st$ as $\ep \to 0$.
Then $X$ is smooth at $0$.
\end{conj}

Strictly speaking, this conjecture would not 
quite provide a link theoretic characterization of smoothness. 
We believe, however, that it should be more accessible than Conjecture~\ref{c:CR}
and could be approached by studying the behavior of the Reeb orbits associated to the
contact structures $\theta_{X,\ep}$ as $\ep$ tends to $0$.


\begin{bibdiv}
\begin{biblist}



\bib{BdFFU15}{article}{
   author={Boucksom, Sebastien},
   author={de Fernex, Tommaso},
   author={Favre, Charles},
   author={Urbinati, Stefano},
   title={Valuation spaces and multiplier ideals on singular varieties},
   conference={
      title={Recent Advances in Algebraic Geometry, a conference in honor of Rob Lazarsfeld's 60th birthday},
   },
   book={
      publisher={London Math. Soc. Lecture Note Series},
   },
   date={2015},
   pages={29--51},
}

\bib{BdM75}{article}{
   author={Boutet de Monvel, Louis},
   title={Int\'egration des \'equations de Cauchy-Riemann induites
   formelles},
   language={French},
   conference={
      title={S\'eminaire Goulaouic-Lions-Schwartz 1974--1975; \'Equations
      aux deriv\'ees partielles lin\'eaires et non lin\'eaires},
   },
   book={
      publisher={Centre Math., \'Ecole Polytech., Paris},
   },
   date={1975},
   pages={Exp. No. 9, 14},
}
%
\bib{Br66a}{article}{
   author={Brieskorn, Egbert},
   title={Examples of singular normal complex spaces which are topological
   manifolds},
   journal={Proc. Nat. Acad. Sci. U.S.A.},
   volume={55},
   date={1966},
   pages={1395--1397},
}

\bib{Br66b}{article}{
   author={Brieskorn, Egbert},
   title={Beispiele zur Differentialtopologie von Singularit\"aten},
   language={German},
   journal={Invent. Math.},
   volume={2},
   date={1966},
   pages={1--14},
}

\bib{dFH09}{article}{
   author={de Fernex, Tommaso},
   author={Hacon, Christopher D.},
   title={Singularities on normal varieties},
   journal={Compos. Math.},
   volume={145},
   date={2009},
   number={2},
   pages={393--414},
}

\bib{DT06}{book}{
   author={Dragomir, Sorin},
   author={Tomassini, Giuseppe},
   title={Differential geometry and analysis on CR manifolds},
   series={Progress in Mathematics},
   volume={246},
   publisher={Birkh\"auser Boston, Inc., Boston, MA},
   date={2006},
}

\bib{DGY16}{article}{
   author={Du, Rong},
   author={Gao, Yun},
   author={Yau, Stephen},
   title={On higher dimensional complex Plateau problem},
   journal={Math. Z.},
   volume={282},
   date={2016},
   number={no.~1-2},
   pages={389--403},
}

\bib{DY12}{article}{
   author={Du, Rong},
   author={Yau, Stephen},
   title={Kohn-Rossi cohomology and its application to the complex Plateau
   problem, III},
   journal={J. Differential Geom.},
   volume={90},
   date={2012},
   number={no.~2},
   pages={251--266},
}

\bib{Elk81}{article}{
   author={Elkik, Ren{\'e}e},
   title={Rationalit\'e des singularit\'es canoniques},
   language={French},
   journal={Invent. Math.},
   volume={64},
   date={1981},
   number={1},
   pages={1--6},
}

\bib{FK72}{book}{
   author={Folland, Gerald},
   author={Kohn, Joseph},
   title={The Neumann problem for the Cauchy-Riemann complex},
   note={Annals of Mathematics Studies, No. 75},
   publisher={Princeton University Press, Princeton, N.J.; University of
   Tokyo Press, Tokyo},
   date={1972},
}

\bib{HL75}{article}{
   author={Harvey, Reese},
   author={Lawson, Blaine},
   title={On boundaries of complex analytic varieties. I},
   journal={Ann. of Math. (2)},
   volume={102},
   date={1975},
   number={2},
   pages={223--290},
}

\bib{Hua06}{article}{
   author={Huang, Xiaojun},
   title={Isolated complex singularities and their CR links},
   journal={Sci. China Ser. A},
   volume={49},
   date={2006},
   number={11},
   pages={1441--1450},
}

\bib{HLY06}{article}{
   author={Huang, Xiaojun},
   author={Luk, Hing-Sun},
   author={Yau, Stephen S. T.},
   title={On a CR family of compact strongly pseudoconvex CR manifolds},
   journal={J. Differential Geom.},
   volume={72},
   date={2006},
   number={3},
   pages={353--379},
}

\bib{Ish13}{article}{
   author={Ishii, Shihoko},
   title={Mather discrepancy and the arc spaces},
   language={English, with English and French summaries},
   journal={Ann. Inst. Fourier (Grenoble)},
   volume={63},
   date={2013},
   number={1},
   pages={89--111},
}

\bib{KR65}{article}{
   author={Kohn, Joseph},
   author={Rossi, Hugo},
   title={On the extension of holomorphic functions from the boundary of a
   complex manifold},
   journal={Ann. of Math. (2)},
   volume={81},
   date={1965},
   pages={451--472},
}

\bib{Loo84}{book}{
   author={Looijenga, Eduard},
   title={Isolated singular points on complete intersections},
   series={London Mathematical Society Lecture Note Series},
   volume={77},
   publisher={Cambridge University Press, Cambridge},
   date={1984},
}

\bib{LY07}{article}{
   author={Luk, Hing Sun},
   author={Yau, Stephen},
   title={Kohn-Rossi cohomology and its application to the complex Plateau
   problem. II},
   journal={J. Differential Geom.},
   volume={77},
   date={2007},
   number={no.~1},
   pages={135--148},
}

\bib{Mat12}{article}{
   author={Mather, John},
   title={Notes on topological stability},
   journal={Bull. Amer. Math. Soc. (N.S.)},
   volume={49},
   date={2012},
   number={4},
   pages={475--506},
}

\bib{McL16}{article}{
   author={McLean, Mark},
   title={Reeb orbits and the minimal discrepancy of an isolated singularity},
   journal={Invent. Math.},
   number={204},
   date={2016},
   pages={505--594},
}

\bib{Mil68}{book}{
   author={Milnor, John},
   title={Singular points of complex hypersurfaces},
   series={Annals of Mathematics Studies, No. 61},
   publisher={Princeton University Press, Princeton, N.J.; University of
   Tokyo Press, Tokyo},
   date={1968},
}

\bib{Mum61}{article}{
   author={Mumford, David},
   title={The topology of normal singularities of an algebraic surface and a
   criterion for simplicity},
   journal={Inst. Hautes \'Etudes Sci. Publ. Math.},
   number={9},
   date={1961},
   pages={5--22},
}

\bib{Nob75}{article}{
   author={Nobile, Augusto},
   title={Some properties of the Nash blowing-up},
   journal={Pacific J. Math.},
   volume={60},
   date={1975},
   number={1},
   pages={297--305},
}

\bib{OZ91}{article}{
   author={Oneto, Anna},
   author={Zatini, Elsa},
   title={Remarks on Nash blowing-up},
   note={Commutative algebra and algebraic geometry, II (Italian) (Turin,
   1990)},
   journal={Rend. Sem. Mat. Univ. Politec. Torino},
   volume={49},
   date={1991},
   number={1},
   pages={71--82 (1993)},
}

\bib{RS93}{article}{
   author={Robbin, Joel},
   author={Salamon, Dietmar},
   title={The Maslov index for paths},
   journal={Topology},
   volume={32},
   date={1993},
   number={4},
   pages={827--844},
}

\bib{Ros64}{article}{
   author={Rossi, Hugo},
   title={Attaching analytic spaces to an analytic space along a
   pseudoconcave boundary},
   conference={
      title={Proc. Conf. Complex Analysis},
      address={Minneapolis},
      date={1964},
   },
   book={
      publisher={Springer, Berlin},
   },
   date={1965},
   pages={242--256},
}

\bib{RT77}{article}{
   author={Rossi, Hugo},
   author={Taylor, Joseph},
   title={On algebras of holomorphic functions on finite pseudoconvex
   manifolds},
   journal={J. Functional Analysis},
   volume={24},
   date={1977},
   number={1},
   pages={11--31},
}

\bib{Sch86}{article}{
   author={Scherk, John},
   title={CR structures on the link of an isolated singular point},
   conference={
      title={Proceedings of the 1984 Vancouver conference in algebraic
      geometry},
   },
   book={
      series={CMS Conf. Proc.},
      volume={6},
      publisher={Amer. Math. Soc., Providence, RI},
   },
   date={1986},
   pages={397--403},
}


\bib{Sem54}{article}{
   author={Semple, John},
   title={Some investigations in the geometry of curve and surface elements},
   journal={Proc. London Math. Soc. (3)},
   volume={4},
   date={1954},
   pages={24--49},
}

\bib{Sev57}{article}{
   author={Severi, Francesco},
   title={Lezioni sulle funzioni analitiche di pi\`u variabili complesse.
   Tenute nel 1956-57 all'Istituto Nazionale di Alta Matematica in Roma},
   language={Italian},
   book={publisher={Cedam-Casa Editrice},},
   volume={XIV},
   date={1957},
}

\bib{Sho02}{article}{
   author={Shokurov, Vyacheslav},
   title={Letters of a bi-rationalist. IV. Geometry of log flips},
   conference={
      title={Algebraic geometry},
   },
   book={
      publisher={de Gruyter, Berlin},
   },
   date={2002},
   pages={313--328},
}

\bib{Tan75}{book}{
   author={Tanaka, Noboru},
   title={A differential geometric study on strongly pseudo-convex
   manifolds},
   note={Lectures in Mathematics, Department of Mathematics, Kyoto
   University, No. 9},
   publisher={Kinokuniya Book-Store Co., Ltd., Tokyo},
   date={1975},
}


\bib{Var80}{article}{
   author={Var{\v{c}}enko, Alexander},
   title={Contact structures and isolated singularities},
   language={Russian, with English summary},
   journal={Vestnik Moskov. Univ. Ser. I Mat. Mekh.},
   date={1980},
   number={2},
   pages={18--21, 101},
}

\bib{Yau81}{article}{
   author={Yau, Stephen},
   title={Kohn--Rossi cohomology and its application to the complex Plateau
   problem. I},
   journal={Ann. of Math. (2)},
   volume={113},
   date={1981},
   number={1},
   pages={67--110},
}

\bib{Yau11}{article}{
   author={Yau, Stephen},
   title={Rigidity of CR morphisms between compact strongly pseudoconvex CR
   manifolds},
   journal={J. Eur. Math. Soc. (JEMS)},
   volume={13},
   date={2011},
   number={1},
   pages={175--184},
}	
	
	
	
	
	

\end{biblist}
\end{bibdiv}
\end{document}